\documentclass[11pt]{article}
\usepackage{geometry}                
\usepackage{amsmath}
\usepackage{amsthm}
\usepackage{graphicx}
\usepackage{amssymb}
\usepackage{epstopdf}
\usepackage{color}
\usepackage{tikz}
\usepackage{multirow}
\usepackage{longtable}
\usepackage[ruled]{algorithm2e}
\theoremstyle{definition}

\theoremstyle{plain}
\newtheorem{theorem}{Theorem}
\newtheorem{prop}{Proposition}
\newtheorem{lemma}{Lemma}

\theoremstyle{remark}

\newcommand{\conv}{\mathbf{conv}}

\newcommand{\Scal}{\mathcal{S}}

\newcommand{\Ncal}{\mathcal{N}}

\newcommand{\diag}{\mathbf{diag}}

\newcommand{\Dcal}{\mathcal{D}}

\title{Relaxing nonconvex quadratic functions by multiple adaptive diagonal perturbations}
\author{Hongbo Dong\thanks{Department of Mathematics, Washington State University, Pullman, WA, 99164, USA}}
\date{Version: March 11, 2014}                                           
\begin{document}
\maketitle
\begin{abstract}
The current bottleneck of globally solving mixed-integer (nonconvex) quadratically constrained problem (MIQCP) is still to construct strong but computationally cheap
convex relaxations, especially when dense quadratic functions are present.
We propose a cutting surface procedure based on multiple diagonal perturbations to derive strong convex quadratic relaxations for nonconvex quadratic problem with separable constraints. Our resulting relaxation does not use significantly more variables than the original problem, in contrast to many other relaxations based on lifting.
The corresponding separation problem is a highly structured semidefinite program (SDP) with convex but non-smooth objective. We propose to solve this separation problem with a specialized primal-barrier coordinate minimization algorithm. Computational results show that our approach is very promising. 
Firstly, our separation algorithm is at least an order of magnitude faster than interior point methods for 
SDPs on problems up to a few hundred variables. Secondly, on nonconvex quadratic integer problems, our cutting surface procedure provides lower bounds of almost the same strength 
with the ``diagonal" SDP bounds used by (Buchheim and Wiegele, 2013) in their branch-and-bound code Q-MIST, while our procedure is at least an order of magnitude faster on problems with dimension greater than 70. 
Finally, combined with (linear) projected RLT cutting planes proposed by (Saxena, Bonami and Lee, 2011), our procedure provides slightly weaker bounds than
their projected SDP+RLT cutting surface procedure, but in several order of magnitude shorter time. Finally we discuss various avenues to extend our work to design more 
efficient branch-and-bound algorithms for MIQCPs.
\end{abstract}
\vspace{0.05in}

\vspace{0.05in}

\section{Introduction}
In this paper we focus on constructing convex quadratic relaxations for the following class of problems,
\begin{equation}\label{eq:qp}
\tag{P}\min_{x\in\Re^n} \  x^T Q x + q^T x \ \ s.t. \ \   x_i \in S_i, \ \forall i \in \{1,..,n\},
 \end{equation}
where $\Re^n$ is the Euclidean space of dimension $n$, $Q$ is indefinite and $S_i \subseteq \Re$ is not necessarily convex. We restrict to this simple structure to simplify discussion and notation in this 
paper. In principle, our approach can be incorporated into branch-and-bound frameworks, such as \cite{FampaLee2013}, to solve general mixed integer nonlinear programing
if the main nonconvexity comes from nonconvex quadratic functions. However, we remark that this formulation (\ref{eq:qp}) already contains
many interesting problem, whose global solution or strong relaxation is of interest in various applications. 


The idea of constructing convex relaxations by diagonal perturbation is not new. The so-called $\alpha$-BB method is a general method to convexify nonlinear functions
by diagonally perturbing their Hessian matrices. See \cite{Skjalwesterlund2012} and references therein. For the important and well-studied case that $S_i = \{0,1\}, \forall i$, 
or equivalently, the Max-Cut problem, the problem (\ref{eq:qp}) can be equivalently reformulated as a binary convex quadratic problem by some diagonal perturbations to $Q$.
 \cite{RendlRinaldiWiegele10} provides a review of solution approaches for this problem and proposes to incorporate some strong SDP-based 
relaxations by solving them using bundle method.
Another relevant line of research \cite{Frangioni_Gentile_2007,  Gunluk_Linderoth_2010, ZhengSunLi2010, DongLinderoth2013} focused on globally solving convex quadratic programming with binary indicator variables, when combined with the so-called perspective constraints, diagonal perturbations are also shown to be very important.

We remark that in many of these approaches, diagonal perturbations are determined using only partial problem information, e.g., Hessian matrices in the problem data and possibly equality constraints. Although in \cite{BillionnetElloumiPlateauQCR} and \cite{ZhengSunLi2010}, the authors determine a single diagonal perturbation by solving a complicated semidefinite program that exploits full problem information, these approaches are computationally very costly and can only be done at the root node in a branch-and-bound tree. After branching,
the problem structure may change and the computed diagonal perturbation may not be useful anymore. On the contrary, our proposed approach in this paper is based on treating diagonal perturbations as cutting surfaces. These cutting surfaces are generated iteratively and adaptively to separate current relaxed solution, hence implicitly we exploit all problem information including other linear and nonlinear inequalities. Furthermore, our separation routine is computationally much cheaper than the SDPs in \cite{BillionnetElloumiPlateauQCR} and \cite{ZhengSunLi2010}, therefore it is possible to revise diagonal perturbations after branching to reflect 
the most updated problem structure.

Many current general purpose global solvers for mixed-integer nonlinear programs are based on the $\alpha$BB and/or the lifting methodology. For the quadratic case, whenever term $x_i x_j$ is present, one introduces the lifted variable $X_{ij}$ with the following so-called RLT constraints
\[
y^{-}_{ij}(x) \leq X_{ij} \leq y^{+}_{ij}(x),
\]
where $L_i \leq x_i \leq R_i$ and 
\[
\begin{aligned}
\ y^{-}_{ij}(x) &:= \max \{R_i x_j + R_j x_i -R_i R_j, L_i x_j + L_j x_i - L_i L_j\},  \\
y^{+}_{ij}(x) &:= \min \{L_i x_j + R_j x_i -L_i R_j, R_i x_j + L_j x_i - R_i L_j\}.
\end{aligned}
\]
We remark that when dense quadratic functions are present, this lifting approach generate a lot of additional variables, which may significantly slow down the whole 
branch-and-bound algorithm.  Anstreicher \cite{Ans07} shows that by combining the SDP constraint with RLT inequalities, one obtains much stronger relaxations than
each of these methods alone, although the resulting SDP+RLT relaxation is generally considered to be computationally expensive. To overcome this difficulty, \cite{SBL08c}
proposes an illuminating approach which generates convex quadratic cutting surfaces by projecting down the SDP+RLT constraints onto the original variable space. Our
work is partially motivated by their work.

Following a different strategy, Burer \cite{BurerChen2013} proposes a competitive branch-and-bound algorithm to solve mixed-binary quadratical constrained programs. Their algorithm is based on an alternative projection augmented Lagrangian algorithm to (approximately) compute the doubly nonnegative relaxations for completely positive reformulations of the original problem. However, this approach is relatively inflexible to represent arbitrary nonconvex structure in $S_i$ (for example, when $S_i$
comprises of integers in a bounded region, or is a union of disjoint intervals), and the bounding algorithm is sensitive to parameter tuning.

Our work is  related to the recent work \cite{Buchheim2013} by Buchheim and Wiegele. In \cite{Buchheim2013}, a branch-and-bound algorithm Q-MIST is designed to 
solve (\ref{eq:qp}) globally, based on solving a ``diagonal" SDP relaxation with interior point methods. Q-MIST is shown to compare favorably to Couenne \cite{couenne}, a general purpose global solver. 
We establish the theoretical connections between our approach and \cite{Buchheim2013}  in Section \ref{sec:equivalence}, and compare the numerical performance in Section \ref{subsec:compBW}.

The full paper is organized as follows. In Section \ref{sec:cutsurf} we derive an iterative cutting surface procedure to construct strong relaxations for (\ref{eq:qp}). In Section \ref{sec:equivalence} we establish the connections between our cutting surface approach and the Buchheim-Wiegele SDP relaxation. Section \ref{sec:sep} is devoted to a 
specialized primal-barrier coordinate minimization algorithm to solve the separation problem in our cutting procedure. Finally, Section \ref{sec:comp} reports our 
numerical results.

\section{Convex cutting surfaces by diagonal perturbations}\label{sec:cutsurf}
For problem (\ref{eq:qp}), consider the following two-dimensional set for each $i$,
\[
D_i := \{ (x_i, x_i^2) | x_i \in S_i\}.
\] 
It is usually possible to fully characterize the convex hull of $D_i$. For simplicity of discussion
we assume that $S_i$ is closed and bounded in this paper, and denote $L_i := \min\{x | x \in S_i\}$ and 
$R_i := \max\{x | x \in S_i\}$. 
Further let us denote  $\ell_i(\cdot)$ the lower convex envelop of $D_i$, i.e., the largest convex function defined on $[L_i, R_i]$,  such that $\ell_i(x) \leq x^2, \forall x\in S_i$,
and $u_i(\cdot)$ the upper concave envelop, i.e., the smallest concave function defined on $[L_i, R_i]$ such that $u_i(x)\geq x^2, \forall x \in S_i$.
We have the following simple characterizations.
\begin{prop}
Let $S_i$, $D_i$, $L_i$, $R_i$, $\ell_i(\cdot)$ and $u_i(\cdot)$ be defined as above, then
\begin{enumerate}
\item $\conv(D_i) = \{(x,y) | \ell_i(x) \leq y \leq u_i(x)\}$;
\item $x^2 \leq \ell_i(x) \leq u_i(x)$, $\forall x \in [L_i, R_i]$;
\end{enumerate}
\end{prop}
\begin{proof} 
We prove (1) first. Since $\conv(D_i)$ is the smallest convex set that contains $D_i$, 
$\conv(D_i) \subseteq \{(x,y) | \ell_i(x) \leq y \leq u_i(x)\}$. To show the opposite inclusion, we assume
otherwise that $\ell_i(\bar{x}) \leq \bar{y} \leq u_i(\bar{x})$ and $(\bar{x},\bar{y}) \notin \conv(D_i)$.
Since $D_i$ is compact,  there exists scalars $a, b$ and $c$ such that at least one of $a$ and $b$ is non-zero, and 
\[
a \bar{x} + b \bar{y} < c, \ and \ a x + b x^2 \geq c, \forall x \in D_i. 
\]
Note the case of $b=0$ implies $\bar{x} \notin [L_i, R_i]$, which contradicts with the implicit assumption 
that $\ell_i(\bar{x})$ and $u_i(\bar{x})$ are well-defined. If $b>0$, we rescale such that $b=1$. 
Then $x^2 \geq -ax + c, \ \forall x \in D_i$ and $\bar{y} < -a\bar{x} + c$. One can then verify that $\hat{\ell}_i(x) := \max(\ell_i(x), -ax+c)$
is a larger convex function such that $\hat{\ell}_i(x) \leq x^2, \ \forall x \in D_i$, which
contradicts with the assumption that $\ell_i(\cdot)$ is the lower convex envelop. The case of $b<0$ is similar.

To prove (2), note that function $x^2$ is a convex function. By the definition of $\ell_i(\cdot)$, 
we must have 
$\max \{ x^2, \ell_i(x) \} \leq \ell_i(x), \forall x\in [L_i, R_i]$, i.e., $x^2 \leq \ell_i(x), \forall x\in [L_i, R_i]$. 
Further $u_i(x) - \ell_i(x)$ is a concave function such that $u_i(x) - \ell_i(x) \geq 0$ for all
$x\in S_i$, including $L_i$ and $R_i$. Therefore we must have $u_i(x) - \ell(x) \geq 0$ for all $x \in [L_i, R_i]$.
\end{proof}

Note that simply studying valid inequalities for the convex hull of the feasible region of (\ref{eq:qp}) 
typically does not provide satisfactory lower bounding approach. For example in the BoxQP case, $S_i = [0,1]$
and $\conv\{ x | x_i \in S_i, i=1,...,n\}$ only provides the trivial box constraints.  To incorporate  
information of the objective function, we first rewrite (\ref{eq:qp}) as a quadratically constrained problem,
\[
\min_{x, v} \ v + q^T x \ \ s.t. \ \ v \geq x^T Q x, \ x_i \in S_i,
\]
and study valid constraints in the space of $(x,v)$.
In this paper we study convex valid constraints obtained by perturbing the quadratic form $x^T Q x$ with separable terms.
Given a vector $d \in \Re^n$, consider the inequality
\begin{equation}\label{eq:kut0}
\begin{aligned}
v \geq & x^T Q x + \sum_{i=1}^n \left( d_i x_i^2 -  d_i y_i (x_i) \right) = x^T (Q+\diag(d))x  - \sum_{i=1}^n  d_i y_i (x_i), \\
\end{aligned}
\end{equation}
where $y_i(x_i)$ is some uni-variate function of $x_i$, whose form possibly depends on the sign of  $d_i$.  We remark that $y_i(x_i)$ can be thought as a ``compensating term" for the perturbation $x_i^2$. Now we consider conditions under which (\ref{eq:kut0}) is valid and convex.
First of all, it is valid if $d_i(x_i^2 -y_i (x_i)) \leq 0, \forall i$. That is,  $y_i (x_i) \geq x_i^2$ when $d_i > 0$ and $y_i (x_i) \leq x_i^2$ when $d_i < 0$.
Secondly, to guarantee the overall convexity, in addition to $Q + \diag(d) \succeq 0$, we require $y_i(x_i)$ to be concave when $d_i > 0$, and convex when 
$d_i < 0$. 
Finally, since it is preferable to have $d_i(x_i^2 - y_i (x_i))$ as large (close to $0$) as possible, natural choices of $y_i(x_i)$ are the lower and upper envelops of $D_i$, i.e.,
\[
y_i(x_i) = \begin{cases}\ell_i(x_i), & d_i < 0; \\ u_i(x_i), & d_i > 0.\end{cases}
\]
Hence we focus on convex valid constraints in the following form, which is parametrized by a vector 
$d$ where $ Q+\diag(d) \succeq 0$,
\begin{equation}\label{eq:kut}
\tag{CUT}v \geq x^T Q x  + \sum_{i: d_i<0} d_i (x_i^2 - \ell_i(x_i)) +  \sum_{i: d_i>0} d_i (x_i^2 - u_i(x_i)).
\end{equation}
Given $(\bar{x},\bar{v})$, the corresponding separation problem is the following convex program
\begin{equation}\label{eq:sep}
\tag{SEP} \begin{aligned}
\inf_{d \in \Re^n} \ & \sum_{i=1}^n g_i(d_i) \\
& Q + \diag(d) \succeq 0,
\end{aligned}
\end{equation}
where $g_i(d_i) \left( := \begin{cases} 
 (\ell_i(\bar{x}_i) - \bar{x}_i^2) d_i, & d_i < 0, \\
(u_i(\bar{x}_i) - \bar{x}_i^2) d_i, &  d_i \geq 0.
\end{cases}\right)$
is a convex function because $\bar{x}_i^2 \leq \ell_i(\bar{x}) \leq u_i(\bar{x})$.
For the convenience of discussion later, we use $\alpha_i$ and $\beta_i$ to denote the corresponding linear coefficients, i.e.,
\[g_i(d_i) := \begin{cases} 
\alpha_i \cdot d_i, & d_i < 0, \\
\beta_i \cdot d_i, &  d_i \geq 0.
\end{cases}\]
Note that we have $0\leq \alpha_i \leq \beta_i$, and the set of optimal solutions to (\ref{eq:sep})
is bounded if and only if $\beta_i > 0, \forall i$.
A feasible vector $d$ defines a valid constraint (\ref{eq:kut}) that cuts off $(\bar{x}, \bar{v})$ 
as long as
\[
\sum_{i} g_i(d_i) < -\bar{v} +\bar{x}^T Q \bar{x}.
\]

We remark that (\ref{eq:sep}) is in a highly structured form. For example, if $\alpha_i = \beta_i, \forall i$, (\ref{eq:sep}) corresponds to the dual problem of the well-know 
Max-Cut problem.

Provided that the univariate functions $\ell_i(\cdot)$ and $u_i(\cdot)$ can be represented in a tractable manner,
for any finite set $\Dcal \subseteq \left\{d \ \middle| \ Q + \diag(d) \succeq 0 \right\}$, the following 
problem is a tractable convex relaxation to (\ref{eq:qp}),
\begin{equation} \label{eq:myrelax}
\tag{DiagR}\begin{aligned}
\mu_{\Dcal} := \min_{v, x} & \ \ \ v + q^T x \\
s.t. & \ \ v \geq x^T Q x + \sum_{i: d_i<0} d_i (x_i^2 - \ell_i(x_i)) +  \sum_{i: d_i>0} d_i (x_i^2 - u_i(x_i)), \ \ \forall d \in \Dcal \\
& \ \ L_i \leq x_i \leq R_i, \ \ i=1,...,n.
\end{aligned}
\end{equation}

With an initial choice of $\Dcal$, we can then iteratively solve (\ref{eq:myrelax}) and update $\Dcal$ by adding a new violated constraint (\ref{eq:kut}) by solving (perhaps a perturbed version of) problem (\ref{eq:sep}). This procedure is summarized in Algorithm \ref{alg:cutsurf}.

\begin{algorithm}[h!]
 \label{alg:cutsurf}
 \SetAlgoLined
 \KwData{$Q \in \Scal^n$, $q\in \Re^n$, and black box routines to evaluate $\ell_i(\cdot)$ and $u_i(\cdot)$;}
 \KwResult{A tractable model (\ref{eq:myrelax}) as a convex relaxation of (\ref{eq:qp}).}
 $\Dcal = \left\{ \lambda\cdot e\right\}$, where $e$ is the all-one vector and $\lambda>|\lambda_{\min}(Q)|$ \;
 \For{$k = 1$ \KwTo $maxIter$ }{
 	Solve (\ref{eq:myrelax}); Let $(\bar{x},\bar{v})$ denote an optimal solution\;
	Compute a feasible vector $d^{new}$ by  (approximately) solving (\ref{eq:sep})\;
	\eIf{(\ref{eq:kut}) with $d=d^{new}$ cuts off $(\bar{x}, \bar{v})$}{
		$\Dcal \leftarrow \Dcal \cup \{d^{new}\}$\;}
	{Terminate\;}
 }
 \caption{A cutting surface algorithm to derive a convex relaxation of (\ref{eq:qp})}
\end{algorithm}

It is worth noting that when $\ell_i(\cdot)$ is relatively complicated and $\Dcal$ has more than one vectors, we can strengthen (\ref{eq:myrelax}) by
introducing variables $y_i$. This makes our procedure a ``partial lifting" procedure. 
\begin{equation} \label{eq:myrelax+}
\tag{DiagR+}\begin{aligned}
\min_{v, x} & \ \ \ v + q^T x \\
s.t. & \ \ v \geq x^T Q x + \sum_{i: d_i<0} d_i (x_i^2 - y_i) +  \sum_{i: d_i>0} d_i (x_i^2 - y_i), \ \ \forall d \in \Dcal \\
& \ \ \ell_i(x_i) \leq y_i \leq u_i(x_i), \ \forall i \\
& \ \ L_i \leq x_i \leq R_i, \ \ \forall i.
\end{aligned}
\end{equation}
However, in all of our computational results later, (\ref{eq:myrelax+}) seems providing same level of bounds with (\ref{eq:myrelax}).

\section{Connection with Buchheim-Wiegele's SDP relaxation}\label{sec:equivalence}
In this section we show that our cutting surface procedure is closely related with a semidefinite relaxation for (\ref{eq:qp}) proposed in \cite{Buchheim2013}, where the authors proposed to globally solve (\ref{eq:qp}) based on solving the following SDP relaxation at each node,
\begin{equation} \label{eq:buchSDP}
\tag{BW}
\begin{aligned}
\mu_{BW} := \min_{x, X} & \ \ \ \langle Q, X\rangle + q^T x \\
s.t. & \ \ \ \ell_i(x_i) \leq X_{ii} \leq u_i(x_i), \\
& \ \ \begin{bmatrix}1 & x^T \\ x& X\end{bmatrix} \succeq 0.
\end{aligned}
\end{equation}
In fact they solve (\ref{eq:buchSDP}) iteratively using interior point methods for SDPs and treat the constraints
\[
\ell_i(x_i) \leq X_{ii} \leq u_i(x_i).
\]
as cutting planes.

We show that our cutting surface procedure is in fact equivalent to (\ref{eq:buchSDP}) in a weak sense. 
First, in Theorem \ref{thm:weak}, we show (\ref{eq:myrelax}) cannot be stronger than (\ref{eq:buchSDP}) for any $\Dcal \subseteq \{d \ |\ Q+\diag(d) \succeq 0\}$.
Then in Theorem \ref{thm:strong}, we show that if for a certain choice of $\Dcal$, (\ref{eq:myrelax}) is strictly weaker than (\ref{eq:buchSDP}), we can cut off current relaxed solution by adding a new vector into 
$\Dcal$.
\begin{theorem}\label{thm:weak}
For any set $\Dcal \subseteq \left\{d \middle| Q + \diag(d) \succeq 0 \right\}$, $\mu_{BW} \geq \mu_{\Dcal}$.
\end{theorem}
\begin{proof}
Note that $\ell_i(x_i) \leq u_i(x_i)$ implies $L_i \leq x_i \leq R_i$, the problem (\ref{eq:buchSDP})
is equivalent to
\begin{equation}\label{eq:BW_reform}
\begin{aligned}
\min_{x, v} \ \ & \ \  v + q^T x \\
s.t. &  \ \ L_i \leq x_i \leq R_i, \\
& \ \ v =  \min_{X} \left\{ \langle Q, X\rangle \ \middle| \  X \succeq xx^T, \ell_i(x_i) \leq X_{ii} \leq u_i(x_i) \right\}.
\end{aligned}
\end{equation}
 It suffices to show that for any $(x,X)$ feasible in (\ref{eq:buchSDP}), and any $d \in \Dcal$
\[
\langle Q, X \rangle \geq x^T Q x + \sum_{i: d_i<0} d_i (x_i^2 - \ell(x_i)) +  \sum_{i: d_i>0} d_i (x_i^2 - u(x_i)).
\]
By re-arranging terms, this inequality is equivalent to 
\[
\langle Q + \diag(d), X-xx^T \rangle - 
\sum_{i: d_i<0} d_i (X_{ii} - \ell(x_i)) -  \sum_{i: d_i>0} d_i (X_{ii} - u(x_i))
\geq 0,
\]
which is valid for any $d \in \Dcal$ as $Q+\diag(d) \succeq 0$.
\end{proof}
\begin{theorem}\label{thm:strong}
Suppose that $\Dcal \subseteq \{d | Q+\diag(d) \succeq 0\}$, $\mu_\Dcal < \mu_{BW}$, and that 
$(\bar{x}, \bar{v})$ is an optimal solution to (\ref{eq:myrelax}), then there exists a new vector 
$\hat{d}$ such that $Q+\diag(\hat{d}) \succeq 0$ and 
\[
\bar{v} < \bar{x}^T Q  \bar{x}  + \sum_{i: \hat{d}_i<0} \hat{d}_i ( \bar{x}_i^2 - \ell( \bar{x}_i)) +  \sum_{i: \hat{d}_i>0} \hat{d}_i ( \bar{x}_i^2 - u( \bar{x}_i)).
\]
\end{theorem}
\begin{proof}
By the reformulation (\ref{eq:BW_reform}), $\mu_\Dcal < \mu_{BW}$ implies that 
\[
\bar{v} -\bar{x}^T Q \bar{x} < \min_{X} \left\{\langle Q, X-\bar{x}\bar{x}^T\rangle \middle| X - \bar{x}\bar{x}^T \succeq 0, \ \ell_i(\bar{x}_i) \leq X_{ii} \leq u_i(\bar{x}_i)\right\}.
\]
Now we derive the Lagrange dual for the minimization on the right hand side. To simplify notation, we 
let $\bar{\ell}_i := \ell_i(\bar{x}_i)$ and  $\bar{u}_i:=u_i(\bar{x}_i)$.
\begin{align*}
\bar{v} -\bar{x}^T Q \bar{x} &< \min_{X}\sup_{\substack{M\succeq 0\\\alpha \geq 0, \beta\geq 0}} 
\langle Q - M, X-\bar{x}\bar{x}^T \rangle  - \sum_{i=1}^n \alpha_i \left(X_{ii} - \ell_i(\bar{x}_i) \right)
-\sum_{i=1}^n \beta_i \left(u_i(\bar{x}_i) - X_{ii}\right) \\
&\leq \sup_{\substack{M\succeq 0\\\alpha \geq 0, \beta\geq 0}} 
\inf_{X} \langle Q - M -\diag(\alpha-\beta), X-\bar{x}\bar{x}^T \rangle -
\sum_{i=1}^n \left[\alpha_i \left(\bar{x}_i^2 - \bar{\ell}_i \right)
+ \beta_i \left(\bar{u}_i - \bar{x}_i^2\right)\right]\\
&= \sup_{\substack{Q-\diag(\alpha-\beta)\succeq 0\\\alpha \geq 0, \beta\geq 0}} -
\sum_{i=1}^n \left[\alpha_i \left(\bar{x}_i^2 - \bar{\ell}_i \right)
+ \beta_i \left(\bar{u}_i - \bar{x}_i^2\right)\right].
\end{align*}
Since the dual problem satisfies the Slater's condition, strong duality holds and the second inequality above
is indeed an equality. Further notice that $\bar{x}_i^2 \leq \bar{\ell}_i \leq \bar{u}_i$, we can assume 
$\min(\alpha_i, \beta_i) = 0, \forall i$ without loss of generality by shifting $\alpha_i$ and $\beta_i$ towards 
zero. Now let $d = \beta - \alpha$. The full inequality implies that there exists $\hat{d} = \hat{\beta} - \hat{\alpha}$
such that $Q + \diag(\hat{d}) \succeq 0$ and 
\[
\bar{v} -\bar{x}^T Q \bar{x} < \sum_{i: \hat{d}_i < 0} \hat{d}_i \left(\bar{x}_i^2 - \ell_i(\bar{x}_i ) \right) + \sum_{i: \hat{d}_i > 0} \hat{d}_i
\left(\bar{x}_i^2 - u_i(\bar{x}_i) \right).
\]
\end{proof}
Note that this result does not necessarily guarantee that the cutting surface algorithm \ref{alg:cutsurf} would generate a sequence of
lower bounds that converges to $\mu_{BW}$. We leave the more detailed analysis for future study while 
focusing on computation in this work.

\section{A Primal-Barrier Coordinate Minimization Algorithm to Solve (\ref{eq:sep})}\label{sec:sep}
To solve (\ref{eq:sep}), it is desirable to use an fast approximate but \textit{strictly feasible} algorithm, i.e., 
we always maintain $d$ such that $Q+\diag(d) \succ 0$. We design a coordinate minimization algorithm for this aim.  Our algorithm is in principle a primal barrier method, 
i.e., we solve the log-det penalty form of (\ref{eq:sep}), and then update the penalty parameter intelligently. Our algorithm is motivated by the so-called ``row-by-row" method
for general SDPs \cite{WenGoldfarb2012}. From now on we assume that $\beta_i > 0$, $\forall i$, and that the optimal solution to (\ref{eq:sep}) is finitely attained. If
this is not the case, we perturb (\ref{eq:sep}) slightly by adding a small positive scalar to all $\alpha_i$ and 
$\beta_i$. We now define the $\log$-$\det$ perturbation to (\ref{eq:sep}) as follows, where $\sigma$ is a positive penalty parameter,
\begin{equation}\label{eq:sep_perturb}
\tag{$\mathbf{SEP}_\sigma$}
\begin{aligned}
\min_{d} \ & \ f(d; \sigma) := \sum_{i=1}^n g_i(d_i) - \sigma \log\det(Q + \diag(d))  \\
 & \  Q + \diag(d) \succ 0.
\end{aligned}
\end{equation}
The sub-differential of $f(d; \sigma)$ is 
\begin{equation}\label{eq:subdiff_f}
\partial f(d; \sigma) = - \sigma \diag\left( \left[Q+\diag(d)\right]^{-1}\right) + \oplus_{i} \partial g_i(d_i),
\end{equation}
where $\oplus_{i} \partial g_i(d_i)$ is the direct product of sub-differentials of $g_i(\cdot)$, which are 
\begin{equation*}
\partial g_i(d_i) = \begin{cases}\alpha_i, & d_i < 0; \\
\left[\alpha_i, \beta_i\right], & d_i = 0; \\
\beta_i, & d_i > 0.
\end{cases}
\end{equation*}
Since the constraint $Q + \diag(d) \succ 0$ defines an open set and cannot be active, the optimality condition of (\ref{eq:sep_perturb}) is
\begin{equation}\label{eq:opt_subd}
0 \in \partial f(d),  \  \ \ Q + \diag(d) \succ 0.
\end{equation}

We solve (\ref{eq:sep_perturb}) in a coordinate minimization manner. In each iteration, we store and update a feasible vector $\bar{d}$ and the
matrix $V = \left[Q + \diag(\bar{d}) \right]^{-1}$.
Motivated by the optimality condition (\ref{eq:opt_subd}),  with a initial feasible vector $\bar{d}$, we choose index $i\in \{1,...,n\}$ with the largest magnitude in the following vector $s(\bar{d})$ to perform the 
minimization,
 \begin{equation}\label{eq:MinNorm_subg}
s\left(\bar{d}\right) := \min\left\{\|u\|_2 \ \middle| \ u \in \partial f\left(\bar{d}; \sigma\right) \right\}, \ \ i = \arg\max_{j} 
\left\{\left|s\left(\bar{d}\right)_j \right|\right\}.
 \end{equation}
Notice that by (\ref{eq:subdiff_f}), $s(d)$ can be evaluated in linear time 
with the information of $V$. With this choice of $i$ we solve the following one-dimensional minimization problem,
\begin{equation}\label{eq:onedim}
\min_{\Delta d_i} \ f(\bar{d}+ \Delta d_i e_i; \sigma) \ \ s.t. \ \ Q + \diag\left(\bar{d}+\Delta d_i e_i\right) \succ 0,
\end{equation}
where $e_i$ is the i-th vector in the canonical basis of $\Re^n$. 
We will later derive a closed form solution to this problem using the problem data and $V$. 
For now we assume $\Delta d_i^*$ is an optimal solution to (\ref{eq:onedim}), then we update $\bar{d}$ by $\bar{d} \leftarrow \bar{d} + \Delta d_i^* e_i$ and $V$  by the Sherman-Morrison formula 
\begin{equation}\label{eq:updateV}
V  \leftarrow V -  \frac{\Delta d_i^*\cdot v_i v_i^T}{1+\Delta d_i^* \cdot V_{ii} },
\end{equation}
where $v_i$ is the i-th column of the previous $V$.

To derive a closed form solution to (\ref{eq:onedim}), we first consider what choices of $\Delta d_i$ guarantee feasibility after the update. 

\begin{lemma} \label{lem:1dimfeas}
Suppose that $\bar{d}$ is a vector such that $ Q + \diag(\bar{d}) \succ 0$ and $V= \left[Q + \diag(\bar{d})\right]^{-1}$, then for each $i$,  
$Q + \diag\left(\bar{d} + \Delta d_i e_i\right) \succ 0$ if and only if $\Delta d_i > -\frac{1}{V_{ii}}$.
\end{lemma}
\begin{proof}
Without loss of generality we assume  $i=n$, and 
\[
Q + \diag(\bar{d}) := \begin{bmatrix}M & q\\q^T & Q_{nn}+\bar{d}_n\end{bmatrix}, \ \ 
V = \left[Q + \diag(\bar{d})\right]^{-1} :=  \begin{bmatrix}\tilde{V} & v_n\\v_n^T & V_{nn}\end{bmatrix}.
\] 
Note that $V_{nn}>0$ as $V \succ 0$. By pre-multiplying $\begin{bmatrix}I & - \frac{v_n}{V_{nn}} \\ 0 & Q_{nn}+\bar{d}_n\end{bmatrix}$ to the equation $V(Q+\diag(\bar{d})) = I$,
we obtain
\begin{equation}\label{eq:mat}
\begin{bmatrix}\tilde{V} - \frac{v_nv_n^T}{V_{nn}} & 0\\ \left(Q_{nn}+\bar{d}_n\right) v_n^T & \left(Q_{nn}+\bar{d}_n\right) V_{nn}\end{bmatrix}
\begin{bmatrix}M & q\\q^T & Q_{nn}+\bar{d}_n\end{bmatrix} = 
\begin{bmatrix}I & - \frac{v_n}{V_{nn}} \\ 0 & Q_{nn}+\bar{d}_n\end{bmatrix}.
\end{equation}
Therefore we have $M^{-1} = \tilde{V} - \frac{v_n v_n^T}{V_{nn}}$. Now by the Schur Complement theorem, 
$Q + \diag\left(\bar{d} + \Delta d_n e_n\right) \succ 0$ if and only if 
\begin{equation}\label{eq:delta_condition}
Q_{nn}+\bar{d}_n + \Delta d_n - q^T M^{-1} q > 0 \Leftrightarrow \Delta d_n > -\left(Q_{nn}+\bar{d}_n\right) + q^T \left( \tilde{V} - \frac{v_n v_n^T}{V_{nn}} \right) q.
\end{equation}
By the upper-right block in (\ref{eq:mat}) and the lower-right block in $V(Q+\diag(\bar{d})) = I$, we have
$\left( \tilde{V} - \frac{v_n v_n^T}{V_{nn}} \right) q = -\frac{v_n}{V_{nn}}$ and 
$v_n^T q + \left(Q_{nn}+\bar{d}_n\right) V_{nn} = 1$, then the condition (\ref{eq:delta_condition}) is equivalent to
\[
\Delta d_n > -\left(Q_{nn}+\bar{d}_n\right) -  \frac{q^T v_n}{V_{nn}}  = -\frac{\left(Q_{nn}+\bar{d}_n\right) V_{nn} + q^T v_n}{V_{nn}} = -\frac{1}{V_{nn}}.
\]
\end{proof}

Now we solve (\ref{eq:onedim}) with the constraint of Lemma \ref{lem:1dimfeas} in mind. 
The sub-differential of $f(\bar{d} + \Delta d_i e_i)$ in (\ref{eq:onedim}) is
\[
\partial g_i(\bar{d}_i + \Delta d_i) - \sigma \left\{\left[Q + \diag(\bar{d}) + \Delta d_i E_{ii}\right]^{-1}\right\}_{ii}
\]
where 
\begin{equation}\label{eq:subg_pieces}
\partial g_i(\bar{d}_i + \Delta d_i) = \begin{cases}\alpha_i, & if \ \Delta d_i < -\bar{d}_i; \\
\left[\alpha_i, \beta_i\right], & if \ \Delta d_i = -\bar{d}_i; \\
\beta_i, & if \ \Delta d_i \geq -\bar{d}_i;
\end{cases}
\end{equation}
and by the Sherman-Morrison formula, 
\begin{equation}\label{eq:smoothcurve}
\sigma\left\{\left[Q + \diag(\bar{d}) + \Delta d_i E_{ii}\right]^{-1}\right\}_{ii} = \sigma\left( V_{ii} -  \frac{\Delta d_i V_{ii}^2}{1+\Delta d_i \cdot V_{ii}} \right)
=\frac{\sigma V_{ii}}{1+\Delta d_i V_{ii}}.
\end{equation}
Then finding a solution to (\ref{eq:onedim}) is equivalent to finding the intersection point between a nonlinear curve (\ref{eq:smoothcurve}) and the piecewise linear curve (\ref{eq:subg_pieces}), with the constraint $\Delta d_i > -\frac{1}{V_{ii}}$ in Lemma \ref{lem:1dimfeas}. Such an intersection point is
guaranteed to exist as $\displaystyle \lim_{\Delta d_i \mapsto +\infty}\frac{\sigma V_{ii}}{1+\Delta d_i V_{ii}} = 0$ and that
 $\beta_i > 0$.
By evaluating (\ref{eq:smoothcurve}) at $\Delta d_i = -\bar{d}_i$,  the solution to (\ref{eq:onedim}) is
\begin{equation}\label{eq:delta_d_value}
\Delta d_i^* = \begin{cases}
\frac{\sigma}{\beta_i} - \frac{1}{V_{ii}}, &  if  \ -\bar{d}_i < -\frac{1}{V_{ii}} \ or \ \sigma\frac{V_{ii}}{1-\bar{d}_i \cdot V_{ii}} > \beta_i; \\
-\bar{d}_i, & if \ -\bar{d}_i \geq -\frac{1}{V_{ii}} \ and \ \alpha_i \leq \sigma\frac{V_{ii}}{1-\bar{d}_i \cdot V_{ii}} \leq \beta_i; \\
\frac{\sigma}{\alpha_i} - \frac{1}{V_{ii}}, &  if \ -\bar{d}_i \geq -\frac{1}{V_{ii}} \ and \ \sigma\frac{V_{ii}}{1-\bar{d}_i \cdot V_{ii}} < \alpha_i.
\end{cases}
\end{equation}
 Figure \ref{fig:cross} illustrates the case of $\alpha_i \leq \frac{\sigma V_{ii}}{1-\bar{d}_i V_{ii}} \leq \beta_i$ and $-\bar{d}_i \geq -\frac{1}{V_{ii}}$, where the intersection takes place at $\Delta d_i^* = -\bar{d}_i$. We further remark that $\alpha_i = 0$ would incur no numerical problem because $-\bar{d}_i \geq -\frac{1}{V_{ii}}$ and 
 $ \sigma\frac{V_{ii}}{1-\bar{d}_i \cdot V_{ii}} < 0$ cannot be simultaneously satisfied (recall that $\sigma > 0 $ and  $V_{ii} > 0$).

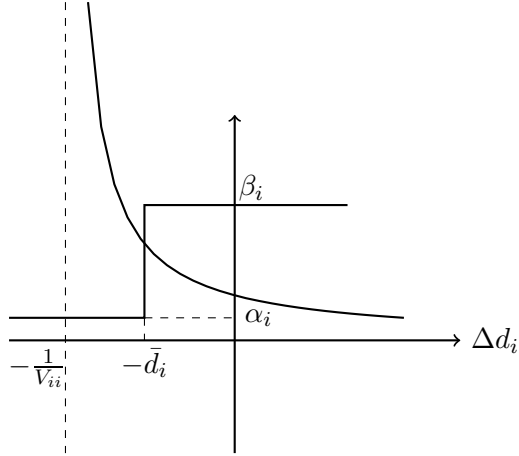
\begin{figure}[htp]
\centering
\begin{tikzpicture}[scale=1.5]

 \draw[ thick, domain=-1.3:1.5] plot (\x, {0.6/(\x+1.5)});
    \draw [->,thick] (0,-1)--(0,2) node (yaxis) [above] {};
    \draw [->,thick] (-2,0)--(2,0) node (xaxis) [right] {$\Delta d_i$};
    
    \draw [dashed] (-1.5, 3) -- (-1.5,-1) ;
    \draw [thick] (-2,0.2) -- (-0.8,0.2) -- (-0.8,1.2) -- (1,1.2);
    \draw [dashed] (-0.8,0.2) -- (0,0.2) node (alpha) [right] {$\alpha_i$};
    \draw [dashed] (-0.8,0.2) -- (-0.8,0);
       \node at (-0.8,-0.2) {$-\bar{d}_i$}; 
        \node at (0.15, 1.35) {$\beta_i$}; 
        
        \node at (-1.75, -0.25) {$-\frac{1}{V_{ii}}$}; 
\end{tikzpicture}
\caption{Illustration of the case when the optimal solution to (\ref{eq:onedim}) is $\Delta d_i^* = -\bar{d}_i$}
\label{fig:cross}
\end{figure}

 As in typical primal barrier algorithms, we update $\sigma$ whenever problem (\ref{eq:sep_perturb}) is solved to some satisfactory precision.
Again we use $s(\bar{d})$ defined in $(\ref{eq:MinNorm_subg})$ as our measure of optimality, and update $\sigma$ 
according to the following rule,
\begin{equation}\label{eq:updatesig}
\sigma \leftarrow  \max(\textsc{sml\_sig}, \textsc{sig\_upd}\cdot \sigma), \ \ if \ \ \frac{s(d)}{\|\beta\|_2} \leq \textsc{subg\_tol}.
\end{equation} 
\textsc{sml\_sig} is a safe-guard parameter to avoid  $\sigma$ to become too small. Our full algorithm to solve (\ref{eq:sep}) is summarized in Algorithm \ref{alg:sep}. Note that the most expensive step in each iteration is 
a single rank one update of $V$, which takes $O(n^2)$ time with a small constant factor.
\begin{algorithm}[h!]
 \label{alg:sep}
 \SetAlgoLined
 \KwData{$Q,  \alpha \in \Re_+^n, \beta \in \Re_{++}^n,  \sigma > 0, \bar{d} \in \Re^n$ such that $Q+\diag(\bar{d})\succ 0$;}
 \KwResult{Vector $\bar{d}$ that is feasible and solves (\ref{eq:sep}) approximately.}
 $V = \left[Q + \diag(\bar{d})\right]^{-1}$ \;
 \For{$k = 1$ \KwTo $maxIter$ }{
 	Compute index $i:=\arg\max_j \left\{\left|s(d)_j\right|\right\}$\;
	Update $\bar{d}_i\leftarrow \bar{d}_i + \Delta d_i^* $ where $\Delta d_i^*$ is computed by (\ref{eq:delta_d_value}) \; 
	Update $V$ using (\ref{eq:updateV}) \;
	Update $\sigma$ using rule (\ref{eq:updatesig}) \;
	Terminate if some termination rule is met\;
 }
 \caption{A primal-barrier coordinate minimization algorithm to solve (\ref{eq:sep})}
\end{algorithm}

\subsection{Implementation Details}
In our implementation and computational experiments, we set $\textsc{sml\_sig}= 10^{-5}$, $\textsc{sig\_upd}=0.8$ and $\textsc{subg\_tol}=0.03$. We choose initial $\bar{d}$ to be $-1.5\lambda_{\min}(Q)$ times the identity matrix, where $\lambda_{\min}(\cdot)$ is the minimal eigenvalue. With $Q$ normalized to have matrix 2-norm 1, initial $\sigma$ is selected to be the median value of the set
\[
\left\{\frac{u_i}{V_{ii}}\right\}_{i=1}^n, \ \ where \ \ u_i \in \partial g_i(\bar{d}_i).
\]
The intuition is that we want the information from $g(\cdot)$ and $\log\det(Q+\diag(\cdot))$  to be ``mixed" at the initial point.
In every $n$ iterations, we check our improvement of the objective value, and terminate our algorithm if the relative 
improvement in last $n$ iterations is less than a parameter \textsc{smll\_prgrss}, which we set at 5E-4. We implement Algorithm \ref{alg:sep} in C language on a Mac OS X system and exploit Apple's Accelerate framework (to vectorize computation) and their implementation of cblas library whenever necessary. We wrap our implementation as a MATLAB mex function to be called within the MATLAB environment in later experiments.

\section{Computational Experiments}\label{sec:comp}
We conduct three numerical experiments to validate our contributions. We implement the cutting surface procedure Algorithm \ref{alg:cutsurf} in the MATLAB environment. 
Convex quadratically relaxations (\ref{eq:myrelax}) are solved using the open source interior point code IPOPT \cite{ipopt} through the MATLAB interface they provided.
Some other separation procedures, i.e., projected RLT cuts \cite{SBL08c} used in our third numerical experiment, are implemented using Yalmip \cite{yalmip} and linear programming 
routines of Gurobi. Sometimes the MATLAB overhead is not negligible, especially when Yalmip is used to prepare inputs to optimization solvers. In these scenarios, we report
only the aggregated time used by optimization solvers only, and remark that the Yalmip overhead can be avoided given a more efficient implementation.

\subsection{Algorithm \ref{alg:sep} versus interior point methods for SDP to solve (\ref{eq:sep})}
In this section, we illustrate by numerical experiments that our Algorithm \ref{alg:sep} can solve 
(\ref{eq:sep}) to moderate precision in significantly shorter time than general purpose interior point algorithm 
for SDPs. 

We generate $Q \in \Scal^n$ with each entry i.i.d  $\Ncal(0,1)$, and then normalized such that 
$\|Q\|_2 = 1$.  $\{\alpha_i\}$ are generated uniformly from $[0,0.5]$ while $\{\beta_i\}$ uniformly from $[0.5,1]$. 
We report the objective values and solver time (wall clock time) used by CSDP (which shows better or equivalent performance than two other interior 
point softwares, SeDuMi and SDPT3 on our instances).  The RelErr column reports the relative differences between objective values 
reported by CSDP and Algorithm \ref{alg:sep} and the final column is the ratio between CSDP time and time used by Algorithm \ref{alg:sep}.
\begin{table}[htdp]
\caption{Comparison between CSDP and Algorithm \ref{alg:sep} to solve (\ref{eq:sep})}
\begin{center}
\begin{tabular}{|c|rr|rr|c|c|}
\hline
\multirow{2}{*}{n}& \multicolumn{2}{c|}{CSDP}&\multicolumn{2}{c|}{Alg \ref{alg:sep}}&\multirow{2}{*}{RelErr}&\multirow{2}{*}{SpeedUp}\\
\cline{2-5}
 & obj & time(s) & obj  & time(s) & &  \\
\hline \hline
50 &   32.35   & 0.108  & 32.36   & 0.002  & 3.8E-4  & 40X \\
100 &  68.30   & 0.355  & 68.32   & 0.014  &  3.8E-4 &  24X \\
200  &135.08 &   1.160  &135.13  &  0.071   & 3.7E-4 &  16X \\
400&  274.36  &   4.813  &274.48 &   0.406   & 4.4E-4 &  12X  \\
\hline
\end{tabular}
\end{center}
\label{default}
\end{table}%

Clearly Algorithm 2 find near optimal strictly feasible solutions to (\ref{eq:sep}) in time at least a magnitude shorter than CSDP. Although
the speed-up ratio decreases as $n$ becomes larger, we remark that nonconvex instances of (\ref{eq:qp}) with $n=50$ are already considered 
difficult for current global solvers.

\subsection{Cutting surface procedure Algorithm \ref{alg:cutsurf} versus Buchheim-Wiegele SDP}\label{subsec:compBW}
In this section we compare our cutting surface procedure Algorithm \ref{alg:cutsurf} with the semidefinite relaxation  
(\ref{eq:buchSDP}) in \cite{Buchheim2013}, on nonconvex integer problems where $S_i = \{-3, ..., 3\}$.  In \cite{Buchheim2013} the authors
developed a branch-and-bound algorithm Q-MIST based on solving semidefinite relaxations (\ref{eq:buchSDP}). It was shown that the Q-MIST algorithm 
compares favorably to Couenne \cite{couenne}, a general purpose global solver for mixed-integer nonlinear programs. However, they only tested instances when
$n\leq 60$. We observe that the cost of solving (\ref{eq:buchSDP}) using interior point methods increases significantly when $n>50$, while our cutting surface procedure
provides the almost identical strength of lower bounds at least one magnitude faster. 
 
 We generate random test instances similarly as in \cite{Buchheim2013}. For the sake of completeness we repeat the settings here. Matrix $Q$ is generated randomly by $Q = \sum_{i=1}^n \mu_i v_i v_i^T$, where
for a percentage $p$ (parameter used to control the level of convexity of $Q$), the first 
$\lfloor pn/100 \rfloor$ number of $\mu_i$ are chosen randomly from $[-1,0]$, and the rest of them 
are chosen randomly from $[0,1]$. Next, each $v_i$ is a random vector of length $n$ with entries independently and uniformly generated from $[-1,1]$, then normalized such that  $\|v_i\|_2 = 1$. Finally the $q$
vector in (\ref{eq:qp}) has all entries uniformly generated from $[-1,1]$. As a baseline for comparison, we run the general purpose global solver BARON
\cite{sahinidis:gamsbaron:12.1.0} for 600 seconds on each instance and record the best upper bound (feasible objective value) and the relative gap. 
We then compare the relaxations for the cases $n=30, 50, 70, 100$, $p=0.2, 0.5, 0.8$ and report our computational results in table \ref{tab:IntQP}. Since in \cite{Buchheim2013} (\ref{eq:buchSDP}) is solved by treating constraints $\ell_i(x_i) \leq X_{ii} \leq u_i(x_i)$ as cutting planes, and the number of cutting planes added is very small, for a more fair comparison, we use the lower bounds provided by the full model (\ref{eq:buchSDP}) while report only the running time of their ``initial model" by replacing
$\ell_i(x_i) \leq X_{ii} \leq u_i(x_i)$ with a single constraint $X_{ii} \leq (L_i + R_i)x_i - L_i R_i.$ (The right hand side is a scalar 9 in our test case.)
All the columns reporting relative gaps are computed by
\[
Gap := \frac{UB-LB}{|UB|} \times 100\%,
\]
where $LB$ is the corresponding lower bound, i.e., solution value of relaxations. The column ``\#it" is the number of iterations used by Algorithm \ref{alg:cutsurf}, and ``$T_{cut}$" is the percentage of time used by the separation procedure (Algorithm \ref{alg:sep}).

\begin{table}[htdp]
\begin{center}
\begin{tabular}{c c||c|r||c|r|r||c|r|c|r|c|}
\multirow{2}{*}{n}& \multirow{2}{*}{p} & \multicolumn{2}{c||}{BARON(600s)} & \multicolumn{3}{c||}{BW-SDP (CSDP)} & \multicolumn{5}{c|}{Algorithm 1 (IPOPT)}\\
\cline{3-12}
& & UB & \multicolumn{1}{c||}{Gap} & LB & \multicolumn{1}{c|}{Gap} & Time & LB & \multicolumn{1}{c|}{Gap} & Time & \#it & $T_{cut}$\\
\hline
\multirow{3}{*}{30} &0.2 &-164.33 & 0.0\% & -186.11 & 13.3\% & \textbf{0.32} & -186.26 & 13.3\% & 0.42 & 11 & 2.2\% \\
&0.5 &-210.28 & 0.0\%  & -235.91 & 12.2\% & \textbf{0.29} & -236.22 & 12.3\% & 0.58 & 12 & 2.0\%\\
&0.8 &-226.28 & 0.0\%& -241.79 & 6.9\% & 0.31 & -242.26 & 7.1\% & \textbf{0.30} & 7 & 2.2\% \\
\hline
\multirow{3}{*}{50} &0.2&-206.11 & 52.5\%& -256.51 & 24.5\% & 1.82 & -257.12 & 24.8\% & \textbf{0.70} & 14 & 4.1\% \\
&0.5 &-345.27 & 43.4\% & -407.59 & 18.1\% & 1.84 & -408.43 & 18.3\% &  \textbf{0.64} & 11 & 4.3\% \\
&0.8 &-407.87 & 13.0\% & -442.63 & 8.5\% & 1.91 & -444.43 & 9.0\% &  \textbf{0.27} & 4 & 3.7\%\\
\hline
\multirow{3}{*}{70} &0.2&-429.50 & 43.0\%& -526.21 & 22.5\% & 8.85 & -528.14 & 23.0\% & \textbf{0.31} & 4 & 6.5\%  \\
&0.5 &-486.38 & 63.8\% & -594.08 & 22.1\% & 7.50 & -595.11 & 22.4\% &  \textbf{1.16} & 15 & 6.4\% \\
&0.8 &-536.57 & 62.9\% & -623.00 & 16.1\% & 8.09 & -624.31 & 16.4\% &  \textbf{0.52} & 7 & 6.9\% \\
\hline
\multirow{3}{*}{100} &0.2&-633.17 & 484.4\% & -820.70 & 29.6\% & 43.40 & -822.43 & 29.9\% &  \textbf{0.90} & 8 & 9.3\% \\
&0.5 &-711.53 & 478.8\% & -829.61 & 16.6\% & 43.65 & -831.26 & 16.8\% &  \textbf{0.98} & 9 & 9.9\% \\
&0.8 &-683.95 & 407.6\% & -855.45 & 25.1\% & 41.45 & -857.24 & 25.3\% &  \textbf{0.94} & 8 & 9.5\% \\
\hline
\end{tabular}
\end{center}
\caption{Lower bounding schemes for randomly generated (\ref{eq:qp}) with $S_i = \{-3, -2, ...,3\}$}
\label{tab:IntQP}
\end{table}%

BARON is able to solve all three instances of $n=30$ to optimality within 600 seconds, while for all other instances, the remaining gaps are
significantly larger than those produced by of BW-SDP and our Algorithm \ref{alg:cutsurf}. In all instances, Algorithm \ref{alg:cutsurf}
provides only slightly weaker bounds than BW-SDP, but in significantly less time when $n\geq 50$. The numbers of iterations
used by Algorithm \ref{alg:cutsurf} remains below 15 for all instances. The percentage of time used by our separation procedure
increases as $n$ increases, but remains below 10\% for all instances.

\subsection{BoxQP instances: Comparison with the projected SDP+RLT approach by Saxena, Bonami and Lee}
In our last numerical experiment, we compare our cutting surface procedure with the projected SDP+RLT procedure proposed in \cite{SBL08c}
for the BoxQP problem, where $S_i = [0,1], \forall i$. We remark that when specialized to BoxQP problems, our procedure is similar to the 
projected SDP+RLT procedure in  \cite{SBL08c}, in the following sense: 
\begin{enumerate}
\item Both Algorithm \ref{alg:cutsurf} and projected SDP+RLT procedure
generate convex quadratic relaxations with multiple quadratic constraints; 
\item Both Algorithm \ref{alg:cutsurf} 
and projected SDP+RLT procedure have an underlying semidefinite relaxation model (BW-SDP versus lifted SDP+RLT relaxation for BoxQP), 
and produce convex quadratic relaxations that are shown to capture most of the strength of corresponding SDP relaxations; 
\item Both Algorithm \ref{alg:cutsurf} and projected SDP+RLT procedure employ a first-order feasible approximate method to generate new cutting 
surfaces (primal-barrier coordinate minimization versus projected subgradient in \cite{SBL08c});
\end{enumerate}
On the other hand, our approach is different from the projected SDP+RLT procedure for the following reasons:
\begin{enumerate}
\item Our procedure exploits more nonconvexity in $S_i$, while the projected SDP+RLT only exploits variable bounds;
\item For the case of BoxQP problems, our  Algorithm \ref{alg:cutsurf} essentially exploit only the diagonal RLT constraints
\[0 \leq X_{ii} \leq x_i, \ \ \forall i,\] 
while ignoring other off-diagonal RLT constraints. Therefore our procedure is theoretically weaker than the (\textbf{ProjSDP}) model in
Theorem 3 of \cite{SBL08c}. However, this loss is remedied by the fact that we can employ a more efficient separation procedure, i.e., Algorithm \ref{alg:sep}, versus
the projected subgradient algorithm in \cite{SBL08c}, which requires an eigenvalue factorization in each iteration.
\end{enumerate}

In order to further exploit the off-diagonal RLT inequalities, we combine the linear cutting plane procedure 
(\textbf{ProjLP}) in \cite{SBL08c} into Algorithm \ref{alg:cutsurf}. We remark that (\textbf{ProjLP}) essentially projects down the full RLT inequalities and 
generates linear valid inequalities in the original variable space by solving some simple linear programs with $O(n^2)$ number of variables, and is computationally very cheap.

Again motivated by the (\textbf{MIQCP-Initial}) reformulation in \cite{SBL08c}, we augment the (\ref{eq:myrelax}) 
model with a convex inequality generated by splitting $Q$ into its convex and concave parts and introducing an additional scalar variable $\tau$,
\begin{equation*}
\begin{aligned}
\min_{v, x} & \ \ \ v + q^T x \\
s.t. & \ \ v \geq x^T Q x + \sum_{i: d_i<0} d_i (x_i^2 - \ell_i(x_i)) +  \sum_{i: d_i>0} d_i (x_i^2 - u_i(x_i)), \ \ \forall d \in \Dcal \\
& \ \ v \geq x^T Q^{+} x + \tau, \\
& \ \ L_i \leq x_i \leq R_i, \ \ i=1,...,n.
\end{aligned}
\end{equation*}
where $Q = Q^{+} + Q^{-}, \  Q^{+} = \sum_{i: \lambda_i > 0} \lambda_i v_i v_i^T$ and $\{(\lambda_i, v_i)\}$ are the eigen-pairs of $Q$. Next we enforce the nonconvex
constraint $\tau \geq x^T Q^{-} x$ by separating the following set by using the methodology of (\textbf{ProjLP}) in \cite{SBL08c},
\[
(x,\tau,v) \in \left\{(x,\tau, v) \middle| \exists X, \begin{array}{l} \langle Q^{+}, X \rangle + \tau - v \leq 0 \\
\langle Q^{-}, X \rangle - \tau \leq 0 \\
L_i \leq x_i \leq R_i , \ \forall i \\
y_{ij}^{-} (x)\leq X_{ij} \leq y_{ij}^{+} (x), \forall i, j\\
\end{array} \right\}
\]
where 
\[
\begin{aligned}
y^{-}_{ij}(x) &= \max \{R_i x_j + R_j x_i -R_i R_j, L_i x_j + L_j x_i - L_i L_j\}, \ \forall i, j \\
y^{+}_{ij}(x) &= \min \{L_i x_j + R_j x_i -L_i R_j, R_i x_j + L_j x_i - R_i L_j\}, \ \forall i, j.
\end{aligned}
\]
We name this augmented procedure ``Alg \ref{alg:cutsurf}+" in our later comparison.

Finally we present our numerical results on all 90 BoxQP instances in \cite{VanNem05b} and compare to the results of ``W3" method reported in \cite{SBL08c}, which corresponds to their implementation of the projected SDP+RLT cutting model (\textbf{ProjSDP}).  (Though we have no information on the specific machine they are using,
it is extremely unlikely that their computer/implementation is several orders of magnitude slower than ours.)
Similar to their comparison strategy, we use the gap between the optimal values and the naive RLT relaxations as a baseline, and calculate how much more gap can be closed by more sophisticated bounding procedures (\textbf{ProjSDP}) and our aforementioned ``Alg \ref{alg:cutsurf}+" procedure. We report our summary in Table \ref{tab:BoxQPsum} and leave the detailed results of each instance in the Appendix. The ``Diff" column is the average difference of the amount of gap closed by these two procedures. A negative number means Alg \ref{alg:cutsurf}+ is worse. We remark that in all instances, Alg \ref{alg:cutsurf}+ is only weaker with a small amount, but requires significantly less time to compute. On the other hand, the difference in time required by these two procedures is several order of magnitude. 
 
\begin{table}[htdp]
\begin{center}
\begin{tabular}{|lc|crc|rc|}
\hline
\multirow{2}{*}{Groups} & \#inst. & \multicolumn{3}{c|}{Average \% gap closed} & \multicolumn{2}{c|}{Average Time (s)} \\
\cline{3-7}
& & SBL & Alg \ref{alg:cutsurf}+ & Diff. & \multicolumn{1}{c}{SBL} &  Alg \ref{alg:cutsurf}+ \\ 
\hline
spar020*-030*&18&97.14\%&94.65\%&-2.49\%&119.73&0.38 \\
spar040* &24&96.37\%&91.51\%&-4.86\%&82.31 &0.46 \\
spar050*-070*&21&93.41\%&89.61\%&-3.80\%&209.92&0.63 \\
spar080*-100*&27&94.24\%&92.89\%&-1.34\%&618.74&0.84 \\
\hline
\end{tabular}
\end{center}
\caption{Summary of comparison with the projected SDP+RLT procedure in \cite{SBL08c} on BoxQP instances}
\label{tab:BoxQPsum}
\end{table}%

We believe the main reason for the huge time difference is that we only search for convex cutting surfaces in a very restricted form, i.e., with Hessian matrices simply diagonal perturbations of the original quadratic function. This restriction greatly simplifies the separation SDP problem one needs to solve. Moreover, this diagonal perturbation approach apparently captures much of problem structure very effectively, e.g., the separability in the constraints $x_i \in S_i, \ \forall i,$ and only small number of iterations are needed
to derive a strong relaxation.

One may argue that like all cutting plane procedures, the SBL procedure has a strong tailing effect. Could it the case that most of the time used by SBL procedure is devoted to closing an insignificant amount of gap? Fortunately, \cite{SBL08c} also reports the time needed to close the
amount of gap that is only 1\% less than the final amount of gap closed, in the columns titled ``W3(Adj)" in many of their tables. We remark that in many instances, Alg \ref{alg:cutsurf}+
provides better bounds than that of ``W3(Adj)", including 7 out of 9 largest instances ``spar100*", while Alg \ref{alg:cutsurf}+ is still several order of magnitude faster (see table \ref{tab:BoxQPlarge}). This clearly 
demonstrates the advantage of Alg \ref{alg:cutsurf}+ over projected SDP+RLT procedures on BoxQP problems, especially on the larger instances.

\begin{table}[htdp]
\begin{center}
\begin{tabular}{|c|rr|rrr|rc|}
\hline
\multicolumn{1}{|c|}{\multirow{2}{*}{Instance}} & \multicolumn{1}{c}{\multirow{2}{*}{RLT}} & \multicolumn{1}{c|}{\multirow{2}{*}{OPT}} & \multicolumn{3}{c|}{\% duality gap closed} & \multicolumn{2}{c|}{Time taken (s)} \\
& & & \multicolumn{1}{c}{W3(Adj)} & \multicolumn{1}{c}{Alg \ref{alg:cutsurf}+} & \multicolumn{1}{c|}{Diff.} & \multicolumn{1}{c}{W3(Adj)} &  \multicolumn{1}{c|}{Alg \ref{alg:cutsurf}+} \\  \hline
spar100-025-1&-7660.75&-4027.50&91.36\%&91.66\%&0.30\%&385.64&1.09 \\
spar100-025-2&-7338.50&-3892.56&91.16\%&91.90\%&0.74\%&321.79&1.55\\
spar100-025-3&-7942.25&-4453.50&92.26\%&91.38\%&-0.88\%&299.23&1.26\\
spar100-050-1&-15415.75&-5490.00&92.62\%&93.88\%&1.26\%&286.59&0.93\\
spar100-050-2&-14920.50&-5866.00&93.13\%&93.50\%&0.37\%&288.09&1.11\\
spar100-050-3&-15564.25&-6485.00&94.81\%&94.49\%&-0.32\%&279.41&0.99\\
spar100-075-1&-23387.50&-7384.20&94.84\%&96.06\%&1.22\%&366.24&0.92\\
spar100-075-2&-22440.00&-6755.50&95.47\%&96.04\%&0.57\%&330.70&1.00\\
spar100-075-3&-23243.50&-7554.00&95.06\%&95.49\%&0.43\%&303.30&1.23\\
\hline
\end{tabular}
\end{center}
\caption{Comparison with the projected SDP+RLT on 9 largest BoxQP instances}
\label{tab:BoxQPlarge}
\end{table}%

\section{Conclusion and Possible Extensions}

We propose a cutting surface procedure based on multiple diagonal perturbations to derive strong but efficiently solvable convex quadratic relaxations for nonconvex quadratic problem with separable constraints $x_i \in S_i, \forall i$.
The corresponding separation problem is a highly structured semidefinite program (SDP) with convex non-smooth objective. We propose to solve the separation problem with a specialized primal-barrier coordinate minimization algorithm. We show that our separation algorithm is at least one order of magnitude faster than interior point method for 
SDPs on problems up to a few hundred variables. On nonconvex quadratic integer problems, our cutting surface procedure provides lower bounds of almost the same strength 
with the SDP bound used by Buchheim and Wiegele  \cite{Buchheim2013} in their branch-and-bound code Q-MIST, while our procedure is at least an order of magnitude faster on problems with dimension greater than 70. 
Combined with linear projected RLT cutting planes proposed in \cite{SBL08c}, our procedure provides slightly weaker bounds than
the projected SDP+RLT cutting surface procedure by Saxena, Bonami and Lee \cite{SBL08c}, but in several order of magnitude shorter time.

There are many avenues to extend our work to devise more effective branch-and-bound algorithms for mixed-integer nonlinear program with nonconvex quadratics. First,
if there are linear equality constraints $Ax = b$, our separation strategy can be revised to exploit this. For example, in (\ref{eq:sep}), $Q+\diag(d)$ only needs to be positive  
semidefinite over the null space of $A$, although computationally care has to be taken to deal with the case that the primal optimal solution is not finitely attained.
Secondly, it is reasonable to expect that when incorporate our diagonal perturbation procedure into a branch-and-bound framework to solve (\ref{eq:qp}) globally, the new algorithm should perform better than Q-MIST, at least on relatively larger instances.
Finally, since our procedure can be thought as a partial lifting procedure that lifts only the diagonal entries $X_{ii}$, and exploiting one-variable valid constraints $\ell_i(x_i) \leq x_i^2 \leq u_i(x_i)$, it would be interesting to identify important multi-variable valid constraints and generalize our approach to a sparse lifting or sparse perturbation approach.

\bibliographystyle{alpha}
\bibliography{./allref}

\begin{thebibliography}{SWMF12}

\bibitem[Ans09]{Ans07}
Kurt~M. Anstreicher.
\newblock Semidefinite programming versus the reformulation-linearization
  technique for nonconvex quadratically constrained quadratic programming.
\newblock {\em Journal of Global Optimization}, 43:471--484, 2009.

\bibitem[BC12]{BurerChen2013}
Samuel~A. Burer and Jieqiu Chen.
\newblock {Globally solving nonconvex quadratic programming problems via
  completely positive programming}.
\newblock {\em Mathematical Programming Computation}, 4:33--52, 2012.

\bibitem[Bel12]{couenne}
Pietro Belotti.
\newblock {COUENNE: a user's manual}.
\newblock Technical report, Department of Mathematical Sciences, Clemson
  University, 2012.

\bibitem[BEP09]{BillionnetElloumiPlateauQCR}
Alain Billionnet, Sourour Elloumi, and Marie-CHristine Plateau.
\newblock Improving the performance of standard solvers for quadratic 0-1
  programs by a tight convex reformulation: The qcr method.
\newblock {\em Discrete Applied Mathematics}, 157:1185--1197, 2009.

\bibitem[BW13]{Buchheim2013}
Christoph Buchheim and Angelika Wiegele.
\newblock {Semidefinite relaxations for non-convex quadratic mixed-integer
  programming}.
\newblock {\em {Mathematical Programming}}, 141:435--452, 2013.

\bibitem[DL13]{DongLinderoth2013}
Hongbo Dong and Jeff Linderoth.
\newblock {On valid inequalities for quadratic programming with continuous
  variables and binary indicators}.
\newblock In {\em The 16th Conference on Integer Programming and Combinatorial
  Optimization; Lecture Notes in Computer Science}, volume 7801, pages
  169--180, 2013.

\bibitem[FG07]{Frangioni_Gentile_2007}
Antonio Frangioni and Claudio Gentile.
\newblock {SDP diagonalizations and perspective cuts for a class of
  nonseparable MIQP}.
\newblock {\em Operations Research Letters}, 35(2):181--185, March 2007.

\bibitem[FLM13]{FampaLee2013}
Marcia Fampa, Jon Lee, and Wendel Melo.
\newblock On global optimization with indefinite quadratics.
\newblock Technical report, Issac Newton Institute Preprint NI13066, 2013.

\bibitem[GL10]{Gunluk_Linderoth_2010}
Oktay G\"{u}nl\"{u}k and Jeff Linderoth.
\newblock Perspective reformulations of mixed integer nonlinear programming
  with indicator variables.
\newblock {\em Mathematical Programming (Series B)}, 124(1-2):183--205, 2010.

\bibitem[L{\"{o}}f04]{yalmip}
J.~L{\"{o}}fberg.
\newblock Yalmip: A toolbox for modeling and optimization in matalb.
\newblock In {\em Proceedings of the CACSD Conference}, Taipei, Taiwan, 2004.

\bibitem[RRW10]{RendlRinaldiWiegele10}
Franz Rendl, Giovanni Rinaldi, and Angelika Wiegele.
\newblock {Solving Max-Cut to optimality by intersecting semidefinite and
  polyhedral relaxations}.
\newblock {\em {Math. Program., Ser. A}}, 121:307--335, 2010.

\bibitem[Sah13]{sahinidis:gamsbaron:12.1.0}
N.~V. Sahinidis.
\newblock {\em {BARON 12.1.0: Global Optimization of Mixed-Integer Nonlinear
  Programs, {\em User's Manual}}}, 2013.

\bibitem[SBL11]{SBL08c}
Anureet Saxena, Pierre Bonami, and Jon Lee.
\newblock Convex relaxations of mixed integer quadratically constrained
  programs: Projected formulations.
\newblock {\em Mathematical Programming, Series A}, 130(2):359--413, 2011.

\bibitem[SWMF12]{Skjalwesterlund2012}
A.~Skj{\"{a}}l, T.~Westerlund, R.~Misener, and C.~A. Floudas.
\newblock {A generalization of the classical $\alpha$BB convex underestimation
  via Diagonal and Nondiagonal Quadratic Terms}.
\newblock {\em {Journal of Optimization Theory and Applications}},
  154:462--490, 2012.

\bibitem[VN05]{VanNem05b}
D.~Vandenbussche and G.~Nemhauser.
\newblock A branch-and-cut algorithm for nonconvex quadratic programs with box
  constraints.
\newblock {\em Mathematical Programming}, 102(3):559--575, 2005.

\bibitem[WB06]{ipopt}
Andreas W{\"{a}}chter and Lorenz~T. Biegler.
\newblock {On the Implementation of a Primal-Dual Interior Point Filter Line
  Search Algorithm for Large-Scale Nonlinear Programming}.
\newblock {\em Mathematical Programming}, 106(1):25--57, 2006.

\bibitem[WGS12]{WenGoldfarb2012}
Zaiwen Wen, Donald Goldfarb, and Katya Scheinberg.
\newblock {Block Coordinate Descent Methods for Semidefinite Programming}.
\newblock In Miguel~F. Anjos and Jean~B. Lasserre, editors, {\em {Handbook on
  Semidefinite, Conic and Polynomial Optimization}}, volume 166 of {\em
  {International Series in Operations Research \& Management Science}}, pages
  533--564. {Springer}, 2012.

\bibitem[ZSL10]{ZhengSunLi2010}
Xiaojin Zheng, Xiaoling Sun, and Duan Li.
\newblock Improving the performance of miqp solvers for quadratic programs with
  cardinality and minimum threshold constraints: A semidefinite program
  approach.
\newblock Manuscript, {Nov.} 2010.

\end{thebibliography}

\newpage

\begin{longtable}{|c|r|r|rrr|rc|}
\caption{Full comparison for BoxQP instances} \\

\hline
\multicolumn{1}{|c|}{\multirow{2}{*}{Instance}} & \multicolumn{1}{c|}{\multirow{2}{*}{RLT}} & \multicolumn{1}{c|}{\multirow{2}{*}{OPT}} & \multicolumn{3}{c|}{\% duality gap closed} & \multicolumn{2}{c|}{Time taken (s)} \\
& & & \multicolumn{1}{c}{SBL} & \multicolumn{1}{c}{Alg \ref{alg:cutsurf}+} & \multicolumn{1}{c|}{Diff.} & \multicolumn{1}{c}{SBL} &  \multicolumn{1}{c|}{Alg \ref{alg:cutsurf}+} \\  
\hline
\endfirsthead

\multicolumn{8}{l}%
{{\tablename\ \thetable{} -- continued from previous page}} \\
\hline
\multirow{2}{*}{Instance} & \multicolumn{1}{c|}{\multirow{2}{*}{RLT}} & \multicolumn{1}{c|}{\multirow{2}{*}{OPT}} & \multicolumn{3}{c|}{\% duality gap closed} & \multicolumn{2}{c|}{Time taken (s)} \\ 
& & & \multicolumn{1}{c}{SBL} & \multicolumn{1}{c}{Alg \ref{alg:cutsurf}+} & \multicolumn{1}{c|}{Diff.} & \multicolumn{1}{c}{SBL} &  \multicolumn{1}{c|}{Alg \ref{alg:cutsurf}+} \\ 
\hline
\endhead

\hline \multicolumn{8}{r}{{Continued on next page}} \\ 
\endfoot

\hline \hline 
\endlastfoot

\hline
spar020-100-1&-1066.00&-706.50&98.28\%&96.85\%&-1.43\%&43.06&0.52 \\
spar020-100-2&-1289.00&-856.50&94.61\%&91.65\%&-2.96\%&2.49&0.33\\
spar020-100-3&-1168.50&-772.00&99.98\%&99.88\%&-0.10\%&408.36&0.54\\
spar030-060-1&-1454.75&-706.00&93.84\%&90.75\%&-3.09\%&13.40&0.19\\
spar030-060-2&-1699.50&-1377.17&97.35\%&95.45\%&-1.90\%&50.79&0.54\\
spar030-060-3&-2047.00&-1293.50&95.62\%&89.67\%&-5.95\%&33.92&0.30\\
spar030-070-1&-1569.00&-654.00&89.88\%&88.99\%&-0.89\%&12.33&0.17\\
spar030-070-2&-1940.25&-1313.00&98.51\%&95.34\%&-3.17\%&188.12&0.50\\
spar030-070-3&-2302.75&-1657.40&96.07\%&94.59\%&-1.48\%&31.57&0.47\\
spar030-080-1&-2107.50&-952.73&95.04\%&90.62\%&-4.42\%&23.57&0.18\\
spar030-080-2&-2178.25&-1597.00&100.00\%&98.73\%&-1.27\%&226.60&0.45\\
spar030-080-3&-2403.50&-1809.78&99.20\%&98.59\%&-0.61\%&339.41&0.42\\
spar030-090-1&-2423.50&-1296.50&99.21\%&96.79\%&-2.42\%&53.39&0.35\\
spar030-090-2&-2667.00&-1466.84&98.56\%&96.10\%&-2.46\%&56.98&0.44\\
spar030-090-3&-2538.25&-1494.00&99.88\%&99.03\%&-0.85\%&565.88&0.36\\
spar030-100-1&-2602.00&-1227.13&98.38\%&95.34\%&-3.04\%&30.28&0.25\\
spar030-100-2&-2729.25&-1260.50&96.93\%&92.34\%&-4.59\%&18.85&0.28\\
spar030-100-3&-2751.75&-1511.05&97.16\%&93.03\%&-4.13\%&56.21&0.53\\
spar040-030-1&-1088.00&-839.50&97.64\%&92.02\%&-5.62\%&117.60&0.79\\
spar040-030-2&-1635.00&-1429.00&91.60\%&74.38\%&-17.22\%&68.46&0.65\\
spar040-030-3&-1303.25&-1086.00&93.04\%&77.32\%&-15.72\%&104.80&0.69\\
spar040-040-1&-1606.25&-837.00&87.85\%&83.55\%&-4.30\%&43.71&0.42\\
spar040-040-2&-1920.75&-1428.00&99.61\%&95.06\%&-4.55\%&114.57&0.50\\
spar040-040-3&-2039.75&-1173.50&92.94\%&88.12\%&-4.82\%&35.77&0.34\\
spar040-050-1&-2146.25&-1154.50&93.71\%&87.38\%&-6.33\%&43.86&0.35\\
spar040-050-2&-2357.25&-1430.98&95.17\%&89.31\%&-5.86\%&54.14&0.36\\
spar040-050-3&-2616.00&-1653.63&94.81\%&89.95\%&-4.86\%&44.05&0.39\\
spar040-060-1&-2872.00&-1322.67&93.47\%&88.65\%&-4.82\%&46.67&0.26\\
spar040-060-2&-2917.50&-2004.23&96.20\%&91.18\%&-5.02\%&80.14&0.52\\
spar040-060-3&-3434.00&-2454.50&99.18\%&97.06\%&-2.12\%&134.80&0.82\\
spar040-070-1&-3144.00&-1605.00&98.85\%&95.72\%&-3.13\%&101.61&0.41\\
spar040-070-2&-3369.25&-1867.50&98.56\%&94.76\%&-3.80\%&94.96&0.37\\
spar040-070-3&-3760.25&-2436.50&97.83\%&94.15\%&-3.68\%&112.96&0.41\\
spar040-080-1&-3846.50&-1838.50&98.43\%&94.72\%&-3.71\%&134.03&0.30\\
spar040-080-2&-3833.00&-1952.50&98.26\%&95.78\%&-2.48\%&47.06&0.24\\
spar040-080-3&-4361.50&-2545.50&97.98\%&96.11\%&-1.87\%&83.80&0.86\\
spar040-090-1&-4376.75&-2135.50&98.22\%&94.45\%&-3.77\%&103.96&0.48\\
spar040-090-2&-4357.75&-2113.00&98.04\%&92.53\%&-5.51\%&83.69&0.33\\
spar040-090-3&-4516.75&-2535.00&99.00\%&97.01\%&-1.99\%&81.20&0.45\\
spar040-100-1&-5009.75&-2476.38&98.72\%&97.14\%&-1.58\%&81.56&0.46\\
spar040-100-2&-4902.75&-2102.50&97.93\%&95.72\%&-2.21\%&121.76&0.41\\
spar040-100-3&-5075.75&-1866.07&95.87\%&94.17\%&-1.70\%&40.16&0.24\\
spar050-030-1&-1858.25&-1324.50&96.40\%&90.23\%&-6.17\%&165.74&0.89\\
spar050-030-2&-2334.00&-1668.00&90.74\%&85.47\%&-5.27\%&79.42&0.50\\
spar050-030-3&-2107.25&-1453.61&91.45\%&83.55\%&-7.90\%&121.65&0.71\\
spar050-040-1&-2632.00&-1411.00&97.23\%&92.86\%&-4.37\%&177.96&0.45\\
spar050-040-2&-2923.25&-1745.76&94.06\%&87.88\%&-6.18\%&85.63&0.40\\
spar050-040-3&-3273.50&-2094.50&97.53\%&93.25\%&-4.28\%&180.96&0.63\\
spar050-050-1&-3536.00&-1198.41&87.88\%&90.36\%&2.48\%&50.22&0.36\\
spar050-050-2&-3500.50&-1776.00&93.13\%&89.00\%&-4.13\%&67.20&0.30\\
spar050-050-3&-4119.75&-2106.10&95.01\%&91.59\%&-3.42\%&93.62&0.36\\
spar060-020-1&-1757.25&-1212.00&91.00\%&85.57\%&-5.43\%&163.42&0.77\\
spar060-020-2&-2238.25&-1925.50&90.22\%&85.51\%&-4.71\%&226.11&1.22\\
spar060-020-3&-2098.75&-1483.00&85.78\%&79.44\%&-6.34\%&121.83&0.45\\
spar070-025-1&-3832.75&-2538.91&92.61\%&87.48\%&-5.13\%&249.97&1.17\\
spar070-025-2&-3248.00&-1888.00&89.79\%&86.47\%&-3.32\%&191.12&0.86\\
spar070-025-3&-4167.25&-2812.28&90.68\%&85.24\%&-5.44\%&214.40&0.83\\
spar070-050-1&-7210.75&-3252.50&94.40\%&92.10\%&-2.30\%&240.93&0.69\\
spar070-050-2&-6620.00&-3296.00&95.77\%&93.53\%&-2.24\%&283.03&0.45\\
spar070-050-3&-7522.00&-4306.50&99.36\%&97.00\%&-2.36\%&693.28&0.46\\
spar070-075-1&-11647.75&-4655.50&96.90\%&96.06\%&-0.84\%&365.50&0.58\\
spar070-075-2&-10884.75&-3865.15&95.57\%&94.45\%&-1.12\%&293.31&0.58\\
spar070-075-3&-11262.25&-4329.40&96.18\%&94.81\%&-1.37\%&342.92&0.56\\
spar080-025-1&-4840.75&-3157.00&93.91\%&89.06\%&-4.85\%&524.07&1.16\\
spar080-025-2&-4378.50&-2312.34&88.14\%&87.17\%&-0.97\%&257.62&0.79\\
spar080-025-3&-5130.25&-3090.88&91.59\%&90.17\%&-1.42\%&420.61&1.17\\
spar080-050-1&-9783.25&-3448.10&92.65\%&92.42\%&-0.23\%&355.97&0.45\\
spar080-050-2&-9270.00&-4449.20&97.50\%&95.21\%&-2.29\%&892.96&0.62\\
spar080-050-3&-10029.75&-4886.00&95.58\%&93.60\%&-1.98\%&435.41&0.55\\
spar080-075-1&-15250.75&-5896.00&96.93\%&96.02\%&-0.91\%&387.48&0.64\\
spar080-075-2&-14246.50&-5341.00&96.95\%&95.72\%&-1.23\%&450.96&0.37\\
spar080-075-3&-14961.50&-5980.50&96.11\%&95.16\%&-0.95\%&416.32&0.54\\
spar090-025-1&-6171.50&-3372.50&90.12\%&88.36\%&-1.76\%&408.73&0.90\\
spar090-025-2&-6015.00&-3500.29&89.45\%&85.12\%&-4.33\%&444.30&0.95\\
spar090-025-3&-6698.25&-4299.00&90.57\%&85.10\%&-5.47\%&446.74&1.16\\
spar090-050-1&-12584.00&-5152.00&95.02\%&93.82\%&-1.20\%&506.72&0.48\\
spar090-050-2&-11920.50&-5386.50&96.61\%&96.15\%&-0.46\%&514.05&0.83\\
spar090-050-3&-12514.00&-6151.00&95.90\%&93.56\%&-2.34\%&991.04&0.45\\
spar090-075-1&-19054.25&-6267.45&95.66\%&95.81\%&0.15\%&462.16&0.62\\
spar090-075-2&-18245.50&-5647.50&95.92\%&95.40\%&-0.52\%&784.59&0.60\\
spar090-075-3&-18929.50&-6450.00&96.11\%&95.87\%&-0.24\%&602.44&0.44\\
spar100-025-1&-7660.75&-4027.50&92.36\%&91.66\%&-0.70\%&670.15&1.09\\
spar100-025-2&-7338.50&-3892.56&92.16\%&91.90\%&-0.26\%&538.03&1.55\\
spar100-025-3&-7942.25&-4453.50&93.26\%&91.38\%&-1.88\%&656.59&1.26\\
spar100-050-1&-15415.75&-5490.00&93.62\%&93.88\%&0.26\%&757.14&0.93\\
spar100-050-2&-14920.50&-5866.00&94.13\%&93.50\%&-0.63\%&929.91&1.11\\
spar100-050-3&-15564.25&-6485.00&95.81\%&94.49\%&-1.32\%&747.46&0.99\\
spar100-075-1&-23387.50&-7384.20&95.84\%&96.06\%&0.22\%&1509.96&0.92\\
spar100-075-2&-22440.00&-6755.50&96.47\%&96.04\%&-0.43\%&936.61&1.00\\
spar100-075-3&-23243.50&-7554.00&96.06\%&95.49\%&-0.57\%&657.84&1.23 \\

\end{longtable}

\end{document}